\documentclass{article}

\usepackage[utf8]{inputenc}
\usepackage[english]{babel}
\usepackage[all]{xy}

\usepackage{scalerel}

\usepackage{amsmath}
\usepackage{amstext}
\usepackage{amsfonts}
\usepackage{amssymb}
\usepackage{amsthm}
\usepackage{amsrefs}
\usepackage{mathtools}
\usepackage{dsfont}
\usepackage{color}
\usepackage{graphicx}
\usepackage[colorinlistoftodos]{todonotes}
\usepackage[colorlinks=true, allcolors=blue]{hyperref}
\usepackage{float}

\allowdisplaybreaks

\makeatletter
\let\@wraptoccontribs\wraptoccontribs
\makeatother

\newtheorem{theorem}{Theorem}

\newtheorem{lemma}[theorem]{Lemma}

\newcommand{\n}{\mathbb{N}}
\newcommand{\norm}[1]{\left\Vert#1\right\Vert}
\newcommand{\T}{{\mathcal{T}}}
 
\newcommand{\la}{\langle}
\newcommand{\ra}{\rangle}
\newcommand{\om}{\omega}

\newcommand{\bC}{\mathbb C} 
\newcommand{\bN}{\mathbb N} 
\newcommand{\cB}{\mathcal B} 
\newcommand{\cH}{\mathcal H} 
\newcommand{\cP}{\mathcal P}

\newtheorem{prop}[theorem]{Proposition}

\newtheorem*{theorem*}{Theorem}
\newtheorem*{prop*}{Proposition}
\newtheorem{corollary}[theorem]{Corollary}

\theoremstyle{definition}
\newtheorem{example}[theorem]{Example}

\newtheorem{remark}[theorem]{Remark}

\renewcommand{\H}{{\rm{\mathbf{H}}}}

\newcommand{\supp}{{\rm supp}}
\newcommand{\Ad}{{\rm Ad}}
\newcommand{\tr}{{\rm tr}}

\renewcommand{\n}{\mathfrak{n}}

\title{More examples of additivity violation of the regularized minimum output entropy in the commuting-operator setup}
\author{Mehrdad Kalantar \thanks{{Department of Mathematics, University of Houston,
  USA, email: mkalantar@uh.edu}} \and Homayoon Shobeiri \thanks{{Department of Mathematics, University of Houston, USA, email: hsshobeiri@uh.edu}}}

\date{}

\begin{document}

\maketitle

\begin{abstract}
\noindent
We generalize recent results of Collins and Youn \cite{CY22}, presenting new classes of quantum channels violating the additivity of the regularized minimum output entropy in the commuting-operator setup.
\end{abstract}


\section*{Introduction}

The additivity problem is among the most fundamental question concerning any notion of entropy for quantum channels. This question asks whether it is possible for given quantum channels $\Phi_1$ and $\Phi_2$ that the entropy of $\Phi_1 \otimes\Phi_2$ be strictly smaller than the sum of entropies of $\Phi_1$ and $\Phi_2$?

A celebrated result of Hastings \cite{Hast09} proves the existence of such examples in the case of the Minimum Output Entropy (see also \cites{HW08, KR01, Sho04}). Hastings' examples were constructed by means of random constructions, and in fact, there is still no concrete deterministic constructions of quantum channels $\Phi_1$ and $\Phi_2$ with $\H_{\min}(\Phi_1 \otimes\Phi_2)< \H_{\min}(\Phi_1)+\H_{\min}(\Phi_2)$. 
(See next section for the relevant definitions.)

A very closely related problem is the additivity of the regularized MOE, which is still completely open. 
We refer the reader to the Introduction of \cite{CY22} for a brief history of these problems and a list of relevant references.

In their recent work \cite{CY22}, Collins and Youn give the very first concrete examples of a pairs of quantum channels violating the additivity of a regularized MOE, however in the infinite-dimensional and commuting-operator setup.
In the finite-dimensional case, the commuting-operator setup is equivalent to the tensor product setup, but these are different in general.

The quantum channels constructed in \cite{CY22} are defined by the unitaries given by the regular representation of the free group. The key technical part of \cite{CY22} is a generalized Haagerup inequality (Lemma 3.1 and Proposition 3.2 op. cit.), for which three different proofs are given \cite[Appendix]{CY22}.

In \cite{Haa79} Haagerup proved an inequality relating the operator norm of the convolution map defined by a finitely supported function on the free group and the $\ell^2$-norm of the function.
Groups who satisfy a similar inequality are said to have the Rapid Decay Property \cite{Jol90}. 

In their Concluding Remarks \cite[Section 5(2)]{CY22}, Collins and Youn note the fact that many other examples of groups with the Rapid Decay Property are known (e.g. all Gromov hyperbolic groups), and state that in those cases, \emph{``it is natural to expect that similar results should hold and will yield other examples of additivity violation phenomena.''}

In this paper, we confirm this by giving new examples of quantum channels demonstrating additivity violation. However, we do not rely on the full power of the Rapid Decay Property. In fact, we prove a general fact (Lemma~\ref{lem:SRD}) which, together with \cite[Lemma 1.4]{Haa79}, gives a very simple proof of \cite[Proposition 3.2]{CY22}. Our analysis also avoids techniques from the more sophisticated theory of operator spaces (c.f. \cite[Appendix]{CY22}).

Our results are applicable to a much larger class of examples. We deduce the additivity violation in the case of free products of arbitrary groups (see Theorem~\ref{cor:free}), whereas the free product of two groups has the Rapid Decay Property if and only if both groups have the Rapid Decay Property \cite{Jol90}.

\section*{Preliminaries}\label{Prel}

Given a Hilbert space $\cH$ and non-zero vectors $v,w\in\cH$, let $\xi_{v,v'}:\cH\to\cH$ denote the rank one operator $\xi_{v,v'}(w) := \langle w,v'\rangle v$ for all $w\in \cH$, and we denote $\xi_v:= \xi_{v,v}$.

A linear map $\xi:\cH\to\cH$ is called \emph{trace class} if there are orthonormal bases $\{v_i\}_{i\in I}$ and $\{w_i\}_{i\in I}$ of $\cH$, and real numbers $c_i\ge 0$ with $\sum_i c_i<\infty$ such that $\xi= \sum_{i\in I}c_i\xi_{v_i, w_i}$. In this case, the trace of $\xi$ is defined by ${\rm tr}(\xi) = \sum_{i\in I}c_i\langle v_i, w_i\rangle$.
We denote by $\T(\cH)$ the space of all trace class operators on $\cH$.
By a \emph{state} on $\cH$ we mean a positive $\xi\in \T(\cH)$ with ${\rm tr}(\xi) = 1$. So, $\xi$ is a state on $\cH$ if and only if there is an orthonormal basis $\{v_i\}_{i\in I}$ of $\cH$, and real numbers $c_i\ge 0$ with $\sum_i c_i=1$ such that $\xi= \sum_{i\in I}c_i\xi_{v_i}$.

If $T$ is a unitary and $\xi\in \T(\cH)$, then $T\xi T^*\in \T(\cH)$, and if $\xi$ is state on $\cH$, then so is $T\xi T^*$. We denote the map $\xi\to T\xi T^*$ by $\Ad(T)$.

A \emph{quantum channel} on $\cH$ is a completely positive map $\Phi:\T(\cH)\to\T(\cH)$ such that ${\rm tr}(\Phi(\xi)) ={\rm tr}(\xi)$ for every $\xi\in\T(\cH)$. In particular, a quantum channel maps states to states.

Suppose that for each $k=1, 2, ..., N$, $T_k:\cH\to\cH$ is a bounded operator such that $\sum_{k=1}^N T_k^*T_k = {\rm Id}_\cH$. Then the map $\Phi:\T(\cH)\to\T(\cH)$ defined by $\Phi(\xi) = \sum_{k=1}^N T_k\xi T_k^*$ is a quantum channel, and its \emph{complementary channel} is given by $\Phi^c:\T(\cH)\to M_N(\bC)$, $\Phi^c(X)= \sum_{k,k'=1}^N \tr(T_k X T_{k'}^*)E_{kk'}$, where $\{E_{ij} : 1\le i,j\le N\}$ is the standard basis of $M_N(\bC)$.

Given a state $\xi$ with eigenvalues $\alpha_1 \geq \alpha_2 \geq ...$ (counting the multiplicity), the \textit{von Neumann entropy} of $\xi$ is $\H(\xi):=-\sum_{i}\alpha_i\log(\alpha_i)$.

Given a quantum channel $\Phi$ on a Hilbert space $\cH$,
the \textit{Minimum Output Entropy} (MOE) of $\Phi$ is
\[
\H_{\min}(\Phi) = \inf \{\H(\Phi({\xi})) \mid \xi\in \T(\cH) \text{ is a state}\} 
\]
and the \textit{regularized MOE} of $\Phi$ is
\[
\overline{\H}_{\min}(\Phi) = \lim_{k\rightarrow \infty} \frac{1}{k}\H_{\min}(\Phi^{\otimes k}) .\\
\]

\noindent
\textit{Preliminaries on groups.}
%
Let $G$ be a countable discrete group. We denote by $\ell^2(G)$ the Hilbert space of square-summable complex functions on $G$, and by $\bC[G]$ the dense subspace of finitely supported complex valued functions on $G$.
Given $g\in G$, for simplicity, we denote $\xi_g=\xi_{\delta_g}$ for the rank-one projection corresponding to the Dirac function $\delta_g\in \ell^2(G)$.

For $g\in G$, define the maps $\lambda_{_G}(g)$ and $\rho_{_G}(g)$ on $\ell^2(G)$ by $\lambda_{_G}(g)(\om)(h) = \om(g^{-1}h)$ and $\rho_{_G}(g)(\om)(h) = \om(hg)$ for $g,h\in G$ and $\om\in\ell^2(G)$.
Then, $\lambda_{_G}(g)$ and $\rho_{_G}(g)$ are unitaries on $\ell^2(G)$ for every $g\in G$, and the maps $\lambda_{_G}:G\to \cB(\ell^2(G))$ and $\rho_{_G}:G\to \cB(\ell^2(G))$ are called the left and the right regular representations of $G$, respectively.

Given a finitely supported $f:G\to\bC$, we define $\lambda_{_G}(f)\in \cB(\ell^2(G))$ by $\lambda_{_G}(f)(\om)=f*\om$, where $f*\om(g)=\sum_{h\in G} f(h)\om(h^{-1}g)$ is the convolution product of $f$ and $\om$.

Let $G$ be a group and $S$ a generating set for $G$ (not containing the neutral element $e$). The \emph{word length} on $G$ (with respect to $S$) is the function $| \cdot | \colon G \to \mathbb{N}_0:=\bN\cup\{0\}$ defined by $|e|= 0$, and for $e\neq g\in G$,
\[
|g| := \min\{n\in \mathbb{N} : \exists s_1, s_2, \dots, s_n \in S \text{ such that } g = s_1 s_2 \dots s_n\} .
\]
The \emph{rank} of a finitely generated group $G$, denoted ${\rm rank}(G)$, is defined as the smallest cardinality of a generating set for $G$. A generating set $S$ for $G$ is said to be \emph{minimal}, if ${\rm rank}(G)=|S|$.

Given a positive integer $m$, we denote by $B_m^G$ the ball of radius $m$, namely the set $\{g\in G : |g| \le m\}$.


\section*{Haagerup-type inequalities for direct products}

As mentioned in the Introduction, the main technical result of the work \cite{CY22} is a generalized Haagerup inequality for direct products of the free group. Inspired by \cite[Lemma 2.1.2]{Jol90}, in the next lemma, we prove a simple, yet very general fact, from which one can deduce Haagerup inequalities for direct products of groups. In particular, the following lemma combined with \cite[Lemma 1.4]{Haa79} immediately implies \cite[Proposition 3.2]{CY22}. We should also note that the proof of \cite[Lemma 3.1]{CY22} also uses \cite[Lemma 1.4]{Haa79} as the base case for an inductive argument.
\begin{lemma}\label{lem:SRD}
Let $G$ and $H$ be discrete groups, $p, q>0$, and $E\subset G$, $F\subset H$ are such that $\|\lambda_{_G}(f)\|\, \le\, p \,\|f\|_{_2}$ for all $f\in\bC [G]$ with ${\rm supp}(f)\subseteq E$, and $\|\lambda_H(f')\|\, \le\, q\, \|f'\|_{_2}$ for all $f'\in\bC [H]$ with ${\rm supp}(f')\subseteq F$.
Then $\|\lambda_{_{G\times H}}(\varphi)\|\, \le\, p q\,\|\varphi\|_{_2}$ for all $\varphi\in\bC [G\times H]$ with ${\rm supp}(\varphi)\subseteq E\times F$.
\end{lemma}

\begin{proof}
Let $\psi\in\bC [G\times H]$. For $s\in H$, define $\varphi_s\in \bC[G]$ by $\varphi_s(a)= \varphi(a, s)$ for all $a\in G$. Define $\psi_s\in \bC[G]$ similarly.
Also, define $\theta, \om\in \bC[H]$ by $\theta(s) = \|\varphi_s\|_{_2}$ and $\om(s) = \|\psi_s\|_{_2}$ for all $s\in H$.
Then ${\rm supp}(\varphi_s)\subseteq E$ for all $s\in H$, and ${\rm supp}(\theta)\subseteq F$. 
Moreover, we have
\[
\|\theta\|_{_2}^{^2} = \sum_{s\in H} \theta(s)^2 = \sum_{s\in H} \|\varphi_s\|_{_2}^{^2} = \sum_{s\in H} \sum_{a\in G} |\varphi(a,s)|^{^2} = \|\varphi\|_{_2}^{^2} ,
\]
and similarly, $\|\om\|_{_2} = \|\psi\|_{_2}$.
Thus,
\begin{align*}
&~\|\varphi*\psi\|_{_2}^{^2} =  
\sum_{g\in G, h\in H} \Big|\sum_{\substack{{a,b\in G, ab=g}\\{s,t\in H, st=h}}}\varphi(a, s) \psi(b, t)\Big|^{^2}
\\=&~
\sum_{h\in H}\Big(\sum_{g\in G} \Big|\sum_{\substack{{s,t\in H}\\ {st=h}}}\sum_{\substack{{a,b\in G}\\{ab=g}}}\varphi_s(a) \psi_t(b)\Big|^{^2}\Big)
=
\sum_{h\in H}\Big(\sum_{g\in G} \Big|\sum_{\substack{{s,t\in H}\\ {st=h}}}(\varphi_s* \psi_t)(g)\Big|^{^2}\Big)
\\=&~
\sum_{h\in H}\Big\|\sum_{\substack{{s,t\in H}\\ {st=h}}}(\varphi_s* \psi_t)\Big\|_{_2}^{^2}
\le
\sum_{h\in H}\Big(\sum_{\substack{{s,t\in H}\\ {st=h}}}\left\|(\varphi_s* \psi_t)\right\|_{_2}\Big)^2
\\\le&~
p^2\sum_{h\in H}\Big(\sum_{\substack{{s,t\in H}\\ {st=h}}}\|\varphi_s\|_{_2} \|\psi_t\|_{_2}\Big)^2
=
p^2\sum_{h\in H}\Big(\sum_{\substack{{s,t\in H}\\ {st=h}}}\theta(s) \om(t)\Big)^2
\\=&~
p^2\|\theta* \om\|_{_2}^{^2}
\le
p^2q^2\|\theta\|_{_2}^{^2} \|\om\|_{_2}^{^2}
=
p^2q^2\|\varphi\|_{_2}^{^2} \|\psi\|_{_2}^{^2} \,
\end{align*}
and this completes the proof.
\end{proof}
Given a group $G$ and $m\le n\in \bN$, we 
denote $\cP_m^n := \{(g_1, \dots, g_n)\in G^n :  \big|\{i : g_i\neq e\}\big| =m \}$.
\begin{corollary}\label{cor:SRD}
Let $G$ be discrete group, let $p>0$ and $E\subset G$ be such that $\|\lambda_{_G}(f)\|\, \le\, p \,\|f\|_{_2}$ for all $f\in\bC [G]$ with ${\rm supp}(f)\subseteq E$.
Then, for every $m\le n\in \bN$ and $\varphi\in\bC [G^n]$ with ${\rm supp}(\varphi)\subseteq E^n\cap \cP_m^n$,
we have
\[
\|\lambda_{_{G^n}}(\varphi)\|\, \le\, {n \choose m}^{\frac12} p^m\,\|\varphi\|_{_2} \,.
\]
\begin{proof}
For every subset $F\subseteq \{1, ..., n\}$, denote $G^F:=\{g=(g_1, \dots, g_n)\in G^n : g_i\neq e \text{ iff } i\in F\}$.
Applying Lemma~\ref{lem:SRD} and Cauchy-Schwarz inequality,
\begin{align*}
\| \lambda_{_{G^n}}(\varphi) \| &= \| \lambda_{_{G^n}}\Big(\sum_{\substack{{F\subseteq \{1, ..., n\}}\\ {|F|=m}}} \varphi\mathds{1}_{G^F}\Big) \| 
\le 
\sum_{\substack{{F\subseteq \{1, ..., n\}}\\ {|F|=m}}} p^m\|\varphi\mathds{1}_{G^F} \|_{_2}
\\&\le
{n \choose m}^{\frac12}(p^m) \Big(\sum_{\substack{{F\subseteq \{1, ..., n\}}\\ {|F|=m}}} \|\varphi\mathds{1}_{G^F} \|_{_2}^{^2}\Big)^{\frac12}
=
{n \choose m}^{\frac12}p^m \,\|\varphi\|_{_2} \,,
\end{align*}
which is the desired inequality.
\end{proof}
\end{corollary}

\section*{Additivity violation of the regularized MOE}

Having the required norm inequalities, we can now follow Collins-Youn's proofs of \cite[Theorem~3.3, Theorem~3.4, Theorem~3.5, Theorem~4.1]{CY22}
to obtain a generalized form of their main result \cite[Theorem 4.1]{CY22}.

Given a group $G$ with a finite generating set $S=\{g_1,\cdots,g_N\}$, define two quantum channels $\Phi_{l,S}, \Phi_{r,S}:\mathcal{T}(\ell^2(G))\to \mathcal{T}(\ell^2(G))$ by
\[
\Phi_{l,S} = \frac{1}{N} \sum_{i=1}^{N} \Ad(\lambda_{_G}(g_i)), ~~~
\Phi_{r,S} = \frac{1}{N} \sum_{i=1}^{N} \Ad(\rho_{_G}(g_i)) .
\]
Define $\n_{\scriptscriptstyle S,G}:= \max_{g\in G\setminus\{e\}}\big|\{(s,t)\in S\times S \,:\, s^{-1}t = g\}\big|$.
For $N\in\bN$, denote $\kappa_{_{N}}:= N^{\frac12}(N^{\frac1N}-1)^{\frac12}-1$.
\begin{theorem}\label{thm:main}
Let G be a group with a finite generating set $S=\{g_1,\cdots,g_N\}$. 
Assume that $\|\lambda_{_G}(f)\| \le \kappa_{\scriptscriptstyle{N}}\,\n_{\scriptscriptstyle S,G}^{^{-\frac{1}{2}}}\, \|f\|_{_2}$ for every $f\in \bC[G]$ with $\supp(f)\subseteq B_2^G$. 
Then
\[
\overline{\H}_{\min}(\Phi_{l,S} \circ \Phi_{r,S}) \lneq \overline{\H}_{\min}(\Phi_{l,S}) + \overline{\H}_{\min}(\Phi_{r,S}) \,.
\]
\end{theorem}

\begin{proof}
Fix $k\in\bN$. Let $\xi\in \mathcal{T}(\ell^2(G^k))$ be a state and set
$\displaystyle X=(x_{g,h})_{g,h\in S^k}=(\Phi_{l,S}^c)^{\otimes k}(\xi)-\frac{1}{N^k}{I}_{N^k}$. 
For the sake of simplicity of notation, we denote $\kappa=\kappa_N$ and $\n=\n_{\scriptscriptstyle S,G}$. 
We have
\begin{align*}
&\mathrm{tr}(X^2) 
=\mathrm{tr}\Big((\xi)\, \big[((\Phi_{l,S}^c)^{\otimes k})^*(X)\big]\Big)
\le
\big\|((\Phi_{l,S}^c)^{\otimes k})^*(X)\big\|
\\&= 
\frac{1}{N^k}\ \Big\|\!\!\sum_{\substack{g,h\in S^k\\ g\neq h}}x_{g,h}\lambda_{_G}({g^{-1}h})\Big\|
\leq
\frac{1}{N^k}\, \sum_{m=1}^{k}\Big\|\!\!\!\sum_{\substack{g,h\in S^k\\ g^{-1}h\in \cP_m^k}} x_{g,h} \lambda_{_G}({g^{-1}h})\Big\|
\\&\leq
\frac{1}{N^k}\, \sum_{m=1}^{k}{k\choose m}^{\frac12} \kappa^{\scriptscriptstyle m}\n^{-\frac{m}{2}}\ \Big\|\!\!\!\sum_{\substack{g,h\in S^k\\ g^{-1}h\in \cP_m^k}} x_{g,h} \delta_{\scriptscriptstyle{g^{\scaleto{-1}{3pt}}}h}\Big\|_{_2}
\\&\leq
\frac{1}{N^k}\, \sum_{m=1}^{k}\bigg[{k\choose m}^{\frac12} \kappa^{\scriptscriptstyle m}\n^{-\frac{m}{2}}N^{\frac{k-m}{2}}\n^{\frac{m}{2}}\,\Big(\!\!\!\sum_{\substack{g,h\in S^k\\ g^{-1}h\in \cP_m^k}} |x_{g,h}|^{^2}\Big)^{\frac12}\bigg]
\\&\leq
\frac{1}{N^k}\, \left(\sum_{m=1}^{k}{k\choose m} \kappa_{_{N}}^{2m}N^{{k-m}}\right)^{\frac12}\Big(\sum_{m=1}^{k}\sum_{\substack{g,h\in S^k\\ g^{-1}h\in \cP_m^k}} |x_{g,h}|^{^2}\Big)^{\frac12}
\\&=
\frac{1}{N^k} \left(\big(\kappa_{_{N}}^2+N\big)^k-N^k\right)^{\frac12}\norm{X}_2 \,,
\end{align*}
%
%
Thus,
\begin{align*}
\left \| (\Phi_{l,S}^c)^{\otimes k}(\xi) \right \|_{_2}
&=
\big \| X +  \frac{1}{N^k}{I}_{N^k}\big \|_{_2} 
\le 
\frac{1}{N^k} \left(\big(\kappa_{_{N}}^2+N\big)^k-N^k\right)^{\frac12} + \frac{1}{N^{\frac{k}{2}}} \ .
\end{align*}
Using the results of \cite{MLDS13} (c.f. the proof of \cite[Theorem 3.5]{CY22}), we get
\begin{align*}
\H_{\min}(\Phi_{l,S}^{\otimes k}) &= \inf_\xi \H((\Phi_{l,S}^c)^{\otimes k}(\xi))
\geq 
\inf_\xi \left (-2\log \left (\left \| (\Phi_{l,S}^c)^{\otimes k}(\xi) \right \|_{_2}\right )\right )\\
&\geq 
-2\log \left ( N^{-\frac{k}{2}}\Big(1+\big[(1+\kappa_{_{N}}^2N^{-1})^k-1\big]^{\frac12}\Big) \right )\\
&=k\log N-2\log \left ( 1+\big[(1+\kappa_{_{N}}^2N^{-1})^k-1\big]^{\frac12} \right )
\end{align*}
and therefore 
\begin{align*}
\overline{\H}_{\min}(\Phi_{l,S}) = \lim_{k\rightarrow \infty} \frac{\H_{\min}(\Phi_{l,S}^{\otimes k})}{k}\geq 
\log N - \frac{1}{2}\log(1+\kappa_{_{N}}^2N^{-1}).
\end{align*}
On the other hand,
\begin{align*}
\overline{\H}_{\min}(\Phi_{l,S}\circ \Phi_{r,S})
&\leq 
\H_{\min}(\Phi_{l,S}\circ \Phi_{r,S})
\leq 
\H((\Phi_{l,S}\circ \Phi_{r,S})(\xi_e))
\\&= 
\H\Big(\frac{1}{N}\xi_{e}+\frac{1}{N^2}\sum_{i,j:i\neq j} \xi_{{g_ig_j^{-1}}}\Big)
\\&\le 
\frac{\log N}{N}+(N^2-N)\cdot \frac{\log N^2}{N^2}= 2\log N-\frac{\log N}{N}.
\end{align*}
Hence
\begin{align*}
&\ \overline{\H}_{\min}(\Phi_{l,S})+\overline{\H}_{\min}(\Phi_{r,S})
=
2\overline{\H}_{\min}(\Phi_{l,S})
\ge
2\log N - \log(1+\kappa_{_{N}}^2N^{-1})
\\&\gneq\
2\log N - \log\left(1+{\big(N(N^{\frac1N}-1)\big)}{N^{-1}}\right)
=
2\log(N)-\frac{\log(N)}{N}
\\&=\
\overline{\H}_{\min}(\Phi_{l,S}\circ \Phi_{r,S}) \,,
\end{align*}
and the proof is complete.
\end{proof}

We see below that under mild conditions on the group $G$, and a generating $S$ for $G$, we get $\n_{\scriptscriptstyle S,G}=1$.

\begin{lemma}\label{lem:n=1}
Let $G$ be a finitely generated group and $S$ a minimal generating set for $G$. Assume that $B_2^G$ contains no element of order two. Then $\n_{\scriptscriptstyle S,G}=1$.
\begin{proof}
Suppose $S=\{g_1, g_2, \dots, g_N\}$ is a minimal generating set for $G$. We show that if $g_i^{-1}g_j = g_k^{-1}g_l\neq e$ for some $1\le i,j,k,l\le N$, then $i=k$ and $j=l$.

For the sake of contradiction, assume $i\neq k$ and $j\neq l$. If also $i\neq l$, then writing $g_i = g_k^{-1}g_lg_j^{-1}$ we see that $g_i \in\la g_k, g_l, g_j\ra$, and therefore $S\setminus \{g_i\}$ is also a generating set for $G$, which contradicts minimality of $S$. Similarly, we conclude that $j\neq k$ leads to a contradiction.
Thus, we must have $g_i = g_l$ and $g_j = g_k$. Then $g_i^{-1}g_j = g_j^{-1}g_i$, which gives $g_i^{-1}g_jg_i^{-1}g_j=e$, and since $g_i^{-1}g_j\neq e$, it follows $o(g_i^{-1}g_j)=2$. This contradicts our assumption that $B_2^G$ contains no element of order two, hence completes the proof.
\end{proof}
\end{lemma}

Another sufficient condition for $\n_{\scriptscriptstyle S,G}=1$ can be formulated in terms of the girth of $G$. 
Recall the girth of a graph is the smallest length of a cycle (and infinite if the graph has no cycles). Given a group $G$ and a generating set $S$ for $G$, we define the \emph{girth} of $G$ with respect to $S$, denoted ${\rm girth}(G, S)$, to be the girth of the Cayley graph of $G$ with respect to $S$.

The proof of the following is straightforward.
\begin{prop}\label{prop:n=1}
Let $G$ be a finitely generated group and $S$ a generating set for $G$. If ${\rm girth}(G, S)\ge 5$, then $\n_{\scriptscriptstyle S,G}=1$.
\end{prop}


\section*{Examples}

In this section, we give a class of examples of groups $G$ satisfying the conditions of Theorem~\ref{thm:main}, hence giving new concrete examples of quantum channels violating the additivity of the regularized minimum output entropy.

This class of groups is formed as the free product of finitely generated groups, with a uniform bound on the cardinality of the sets of generators.
For this, we will prove a Haagerup type inequality for free products. Similar inequalities were proven by Jolissaint in \cite{Jol90} where the Rapid Decay property was proven for the free product of two groups with the same property. Although the inequalities we need are very similar to that of Jolissaint's, but we cannot directly use them in their stated form in \cite{Jol90}. First, we do need rather a more precise understanding of the constants involved in inequalities, whereas in \cite{Jol90} only existence of some constants are proven.
Furthermore, the stated result in \cite{Jol90} is for the free product of two groups. This is of course enough to establish Haagerup type inequalities for the free product of arbitrary finite number of groups by induction, but the constants grow (exponentially) in terms of the number of groups, which creates an issue for our desired estimates.
We therefore prove a modified version of \cite[Lemma 2.2.4]{Jol90}, which is needed to guarantee that the conditions of Theorem~\ref{thm:main} are satisfied for the free product of arbitrary large number of groups.

Let us briefly review the definition, and fix some notations.
For each $i=1, \dots, N$, let $G_i=\langle S_i \mid R_i\rangle$ be a group defined by the generating set $S_i$ and the relations set $R_i$. Then the free product $G = G_1* \cdots * G_N$ is defined as $G=\langle S_1\cup\cdots\cup S_N \mid R_1\cup\cdots\cup R_N\rangle$.
Every non-neutral element $g\in G = G_1* \cdots * G_N$ has a unique reduced decomposition $g= \gamma_1\cdots \gamma_n$ where $n\in\bN$, $e\neq \gamma_k\in G_{i_k}$, and $G_{i_k}\neq G_{i_{k+1}}$ for all $1\le k \le n-1$.
In this case we say $g$ has the \emph{reduced length} $n$ and write $\ell(g)= n$. For each $1\le k \le n$, we define $\pi_k(g)=\gamma_k$. Given $l\in \bN_0$, we denote $\Lambda_l :=\{g\in G_1* \cdots * G_N \,\mid\, \ell(g)= l\}$ and for $k\le l$ and $1\le i\le N$, $\Lambda_l^{(k,i)} :=\{g\in \Lambda_l \,\mid\, \pi_k(g)\in G_i\}$.

If each $G_i$, $1\le i\le N$, is finitely generated (and $S_i$ is chosen to be finite), unless otherwise stated, we consider the word length on $G_i$ defined by $S_i$, and the one on $G$ defined by $S=S_1\cup\cdots\cup S_N$.
\begin{theorem}\label{thm:SRD-free}
Let $N\in \bN$, and let $G= G_1* G_2 *\cdots * G_N$ be the free product of groups $G_1, G_2, \dots, G_N$. For each $i=1, \dots, N$, let $p_i$ be a positive real such that $\|\lambda_{_{G_i}}(f)\|\, \le\, p_i \,\|f\|_{_2}$ for all $f\in\bC [G_i]$ with ${\rm supp}(f)\subseteq B_2^{G_i}$. 
Then
\begin{equation}\label{eq:main}
    \|\lambda_{_{G}}(\varphi)\|\, \le\, 5\sqrt{2} \,\max_{1\le i\le N}\{p_i\}\,\|\varphi\|_{_2}
\end{equation}
for all $\varphi\in\bC [G]$ with ${\rm supp}(\varphi)\subseteq B_{2}^G$. 
\end{theorem}

\begin{proof}
We follow the line of arguments originally due to Haagerup \cite{Haa79}, which was also followed in \cite{Jol90}.
Let $\varphi\in\bC [G]$ with ${\rm supp}(\varphi)\subseteq B_{2}^G$. 

First, assume furthermore that ${\rm supp}(\varphi)\subseteq \Lambda_1$. 
Let $m,l\in \mathbb{N}_0$ and let $\psi$ be supported on $\Lambda_l$. We show that 
\begin{equation}\label{eq1}
\|(\varphi*\psi)\, \mathds{1}_{\Lambda_m}\|_{_2}\leq \max_{1\le i\le N}\{p_i\}\,\|\varphi\|_{_2}\, \|\psi\|_{_2}
\end{equation}
if $l-1\leq m\leq l+1$. 
Obviously, $(\varphi*\psi)\, \mathds{1}_{\Lambda_m}=0$ for other choices of $m$.

If $m=l+1 \text{ or } l-1$, then as noted in the proof of \cite[Lemma~2.2.4]{Jol90}, an argument just as in \cite[Lemma~1.3]{Haa79} yields
$\|(\varphi*\psi)\, \mathds{1}_{\Lambda_m}\|_{_2}\le \|\varphi\|_{_2}\, \|\psi\|_{_2}$.

Now, assume $m=l$. For every $b\in \Lambda_{l-1}$, define $\psi^b\in \bC[G]$ by $\psi^b(\gamma)=\mathds{1}_{\Lambda_1}(\gamma)\psi(\gamma b)$. 
Then, for each $1\le i \le N$, we have
\begin{align*}
&~\big\|(\varphi*\psi)\,\mathds{1}_{\Lambda_m^{(1,i)}}\big\|_{_2}^{^2}
\,= \sum_{\gamma\in \Lambda_m^{(1,i)}} \big|\varphi*\psi(\gamma) \big|^{^2}
= 
\sum_{b\in \Lambda_{l-1}}\big\|\varphi\mathds{1}_{G_i}*\psi^{b}\mathds{1}_{G_i}\big\|_{_2}^{^2}
\\\le&
\sum_{b\in \Lambda_{l-1}}p_i^2~\big\|\varphi\mathds{1}_{G_i}\big\|_{_2}^{^2}~\big\|\psi^{b}\mathds{1}_{G_i}\big\|_{_2}^{^2}
\, =\,
p_i^2~\big\|\varphi\mathds{1}_{G_i}\big\|_{_2}^{^2}~\big\|\psi\,\mathds{1}_{\Lambda_m^{(1,i)}}\big\|_{_2}^{^2}\,.
\end{align*}
Thus,
\begin{align*}
&~\big\|(\varphi*\psi)\,\mathds{1}_{\Lambda_m}^{}\big\|_{_2}^{^2}
=
\sum_{i=1}^N\big\|(\varphi*\psi)\,\mathds{1}_{\Lambda_m^{(1,i)}}\big\|_{_2}^{^2}
\\\le& ~
\sum_{i=1}^Np_i^2~\big\|\varphi\,\mathds{1}_{G_i}\big\|_{_2}^{^2}~\big\|\psi\,\mathds{1}_{\Lambda_m^{(1,i)}}\big\|_{_2}^{^2}
\le 
(\max_{1\le i\le N}\{p_i^2\})\, \|\varphi\|_{_2}^{^2}~\|\psi\|_{_2}^{^2}.
\end{align*}
Next, assume that ${\rm supp}(\varphi)\subseteq \Lambda_2$. We prove the inequality~\eqref{eq1} for $\psi\in\bC [G]$ supported in $\Lambda_l$ and $|l-2|\le m \le l+2$, and we observe $(\varphi*\psi)\, \mathds{1}_{\Lambda_m}=0$ for other choices of $m$.

For $m=l+2, l, l-2$, again similarly as in \cite[Lemma~1.3]{Haa79} we get
\[
\|(\varphi*\psi)\, \mathds{1}_{\Lambda_m}\|_{_2}\leq \|\varphi\|_{_2}\ \|\psi\|_{_2}\,.
\]
%
Now suppose $m=l+1$.
For each $a\in \Lambda_1$, define $\varphi_a\in \bC[G]$ by $\varphi_a(\gamma) =\mathds{1}_{\Lambda_1}(\gamma)\varphi(a\gamma)$. 
Then 
\begin{align*}
\big\|(\varphi*\psi)\,\mathds{1}_{\Lambda_m}\big\|_{_2}^{^2}
&=\sum_{i=1}^N \sum_{\gamma\in \Lambda_m^{(2,i)}} \big|\varphi*\psi(\gamma) \big|^{^2}
\\&\le \sum_{i=1}^N \sum_{a\in \Lambda_1}\sum_{b\in \Lambda_{l-1}}\big\|\varphi_{a}\mathds{1}_{G_i}*\psi^{b}\mathds{1}_{G_i}\big\|_{_2}^{^2}
\\&\le \sum_{i=1}^N\sum_{a\in \Lambda_1}\sum_{b\in \Lambda_{l-1}} p_i^2~\big\|\varphi_{a}\mathds{1}_{G_i}\big\|_{_2}^{^2}~\big\|\psi^{b}\mathds{1}_{G_i}\big\|_{_2}^{^2}
\\&\le \max_{1\le i \le N}\{p_i^2\}~\sum_{a\in \Lambda_1} \|\varphi_{a}\|_{_2}^{^2}\sum_{b\in \Lambda_{l-1}}\|\psi^{b}\|_{_2}^{^2}
\\&= \max_{1\le i \le N}\{p_i^2\}~\|\varphi\|_{_2}^{^2}~\|\psi\|_{_2}^{^2} .
\end{align*}
%
Finally, assume $m=l-1$. 
Define the functions $\varphi', \psi'\in \bC[G]$ by $\varphi'(s):=\mathds{1}_{\Lambda_1}(s)\|\varphi_s\|_{_2}$ 
and $\psi'(t):=\mathds{1}_{\Lambda_{l-1}}(t)\|\psi^t\|_{_2}$. 

Since $\supp(\varphi')\subseteq \Lambda_1$ and $\supp(\psi')\subseteq \Lambda_l$ and $m=l+1$, the previous case gives $\|(\varphi'*\psi')\,\mathds{1}_{\Lambda_m}\|_{_2}\le\max_{1\le i \le N}\{p_i\} ~\|\varphi'\|_{_2}~\|\psi'\|_{_2}$.

We observe that $\|\varphi'\|_{_2}^{^2} = \sum_{s\in \Lambda_1}\sum_{a\in \Lambda_1}|\varphi(sa)|^{^2} = \|\varphi\|_{_2}$,
and similarly $\|\psi'\|_{_2} =\|\psi\|_{_2}$.
Moreover, for every $\gamma=\gamma_1\gamma_2\cdots \gamma_m\in \Lambda_m$, we have
\begin{align*}
\big|\varphi*\psi(\gamma)\big| 
&= 
\Big|\sum_{st=\gamma}\varphi(s)\psi(t)\Big| 
= \Big|\sum_{a\in \Lambda_1}\varphi(\gamma_1a)\psi(a^{-1}\gamma_2\cdots \gamma_m) \Big|
\\&\le 
\Big(\sum_{a\in \Lambda_1}|\varphi(\gamma_1a)|^{^2}\Big)^{\frac12}\Big(\sum_{a\in \Lambda_1}\psi(a^{-1}\gamma_2\cdots \gamma_m)\Big)^{\frac12}
\\&=
\varphi'(\gamma_1)\psi'(\gamma_2\cdots \gamma_m)
=
\varphi'*\psi'(\gamma) ,
\end{align*}
which shows $\big|(\varphi*\psi)\mathds{1}_{\Lambda_m}\big|\le (\varphi'*\psi')\,\mathds{1}_{\Lambda_m}$.
Thus,
\begin{align*}
\|(\varphi*\psi)\mathds{1}_{\Lambda_m}\|_{_2}
&\le
\|(\varphi'*\psi')\mathds{1}_{\Lambda_m}\|_{_2}
\\&\le
\max_{1\le i \le N}\{p_i\} \|\varphi'\|_{_2}\|\psi'\|_{_2}
=
\max_{1\le i \le N}\{p_i\} \|\varphi\|_{_2}\|\psi\|_{_2}\, .
\end{align*}
For $j=1,2$, let $\varphi_j=\varphi\mathds{1}_{\Lambda_j}$.
Let $\psi\in \bC[G]$, and for each $l\in \bN_0$, let $\psi_l:=\psi\mathds{1}_{\Lambda_l}$.
Then
\begin{align*}
\|\varphi*\psi\|_{_2}^{^2}
&=
\sum_{m\in\bN}\,\big\|(\varphi*\psi)\,\mathds{1}_{\Lambda_m}\big\|_{_2}^{^2}
\le
\sum_{m\in\bN}\,\Big(\sum_{j=1,2}\,\sum_{l\in\bN_0}\big\|(\varphi_j*\psi_l)\,\mathds{1}_{\Lambda_m}\big\|_{_2}\Big)^2
\\&\le
\max_{1\le i \le N}\{p_i^2\}\, \sum_{m\in\bN}\,\Big(\sum_{j=1,2}\,\sum_{l= |m-2|}^{m+2}\|\varphi_j\|_{_2}\,\|\psi_l\|_{_2}\Big)^2
\\&=
\max_{1\le i \le N}\{p_i^2\} \,\Big(\sum_{j=1,2}\|\varphi_j\|_{_2}\Big)^2~\sum_{m\in\bN}\,\left(\sum_{k=0}^{2\min\{m,2\}}\|\psi_{m+2-k}\|_{_2}\right)^2
\\&\le
{2}\cdot 5\cdot \max_{1\le i \le N}\{p_i^2\} \,\|\varphi\|_{_2}^{^2}~\sum_{m\in\bN}\,\sum_{k=0}^{2\min\{m,2\}}\|\psi_{m+2-k}\|_{_2}^{^2}
\\&\le
50 \max_{1\le i \le N}\{p_i^2\} \,\|\varphi\|_{_2}^{^2}~\sum_{m\in\bN}\|\psi_{m}\|_{_2}^{^2}
=
50 \max_{1\le i \le N}\{p_i^2\} \,\|\varphi\|_{_2}^{^2}\,\|\psi\|_{_2}^{^2} ~,
\end{align*}
and this completes the proof.
\end{proof}

In order to apply Theorem~\ref{thm:main} to free product groups $G= G_1* G_2 *\cdots * G_N$, we need to relate the constants $\n_{\scriptscriptstyle S,G}$ to those of its factor groups $\n_{\scriptscriptstyle S_i,G_i}$.
\begin{lemma}\label{lem:n-free-prod}
For each $i=1, 2, \dots, N$,  let $G_i$ be a group and $S_i$ a generating set for $G_i$.
Let 
$G= G_1* G_2 *\cdots * G_N$ 
and let $S=\cup_{i=1}^N S_i$.
Then $\n_{\scriptscriptstyle \cup S_i, *_i G_i}= \max_i\{\n_{\scriptscriptstyle S_i, G_i}\}$.
\begin{proof}
It suffices to prove the claim for $N=2$, the general case follows by induction. 

We need to show that if for some $a,b,c,d\in S_1\cup S_2$ we have $a^{-1}b=c^{-1}d$, then either the equality is trivial, or that 
all elements $a, b, c, d$ belong to the same group, they are  either all in $S_1$ or all in $S_2$.

So, suppose $ca^{-1}bd^{-1}=e$. Assume $c\in G_1$. In this case, if $a\in G_2$, then by the uniqueness of the reduced forms, we must have $a=b$ and $c=d$. If $c,a\in G_1$ and $b\in G_2$, again we must have $a=c$ and $b=d$. Finally, if $c,a, b\in G_1$, similarly it follows $d\in G_1$. The case where $c\in G_2$ is similar.
\end{proof}
\end{lemma}

\begin{theorem}\label{cor:free}
Suppose $M, N\in\bN$ are such that $e^{64(M^2+M+1)}\le N$, and suppose that for each $i=1, ..., N$, $G_i$ is a group with ${\rm rank}(G_i)\le M$. For each $i=1, ..., N$, let $S_i\subset G_i$ be a generating set with $|S_i| = {\rm rank}(G_i)$, and assume that for each $i=1, ..., N$, either $B_2^{\scriptscriptstyle G_i}\cap\{\gamma\in G_i : o(\gamma)=2\} =\emptyset$ or that ${\rm girth}(G, S)\ge 5$.
Let $G= G_1* G_2 *\cdots * G_N$, and $S=\cup_{i=1}^{N}S_i$. Then
\[
\overline{\H}_{\min}(\Phi_{l,S} \circ \Phi_{r,S}) \lneq \overline{\H}_{\min}(\Phi_{l,S}) + \overline{\H}_{\min}(\Phi_{r,S}) .
\]
\end{theorem}

\begin{proof}
For every $1\le i\le N$, we have
$\big| B_2^{G_i} \big|\le M^2+M+1$, 
and so for every $f\in\bC[G_i]$ with $\supp(f)\subset B_2^{G_i}$  
we have $\|\lambda_{_{G_i}}(f)\|\le (M^2+M+1)^{\frac12}\|f\|_{_2}$ by the triangle inequality. 
Thus, by Theorem~\ref{thm:SRD-free}, for every $\varphi\in\bC[G]$ with $\supp(\varphi)\subset B_2^{G}$ 
we have $$\|\lambda_{_{G}}(\varphi)\|\le 5\sqrt{2}\,(M^2+M+1)^{\frac12}\|\varphi\|_{_2}\,.$$
Since
\begin{align*}
5\sqrt{2}\,(M^2+M+1)^{\frac12} + 1
&\le
8\,(M^2+M+1)^{\frac12}
\\&\le
\sqrt{\log N}
\\&\le 
N^{\frac12}(N^{\frac1N}-1)^{\frac12}\, .
\end{align*}
Thus, it follows $$5\sqrt{2}\,(M^2+M+1)^{\frac12} \le \kappa_{_{N}} \leq \kappa_{_{|S|}}\,.$$ 
The combination of Lemma~\ref{lem:n=1}, Proposition~\ref{prop:n=1}, and Lemma~\ref{lem:n-free-prod}, implies $\n_{\scriptscriptstyle S, G}=1$. Hence the result follows from Theorem~\ref{thm:main}.
\end{proof}



\begin{example}
In Theorem~\ref{cor:free} above, for any choice of numbers $M, N$ as in the statement, each group $G_i$ can be taken to be either a torsion free group generated by at most $M$ elements, or any finite group of odd order $5\le o(G_i)\le M$, or the finite cyclic group $\mathbb{Z}_n$ for $5\le n\le M$, since ${\rm girth}(\mathbb{Z}_n, \{\bar 1\})=n$.
\end{example}

\begin{remark}
Given finitely generated groups $G_1$, \dots, $G_N$, and a common finite subgroup $H$ of all $G_i$'s, then the proof of Theorem~\ref{thm:SRD-free} can be modified slightly to obtain a similar inequality as in~\eqref{eq:main}, for the free product of $G_i$'s amalgamated over $H$, with an additional constant depending on the cardinality of $H$. In particular, in Theorem~\ref{cor:free} we can allow amalgamation over a common finite subgroup. 
\end{remark}

\subsection*{Acknowledgements}
MK was supported by the Simons Foundation Collaboration Grant \#713667 and by the NSF Grant DMS-2155162.

\end{document}